\documentclass[12pt,twoside,leqno,final]{amsart}
\usepackage{amsmath}
\usepackage{amssymb}

\theoremstyle{plain}
\newtheorem{thm}{Theorem}[section]
\newtheorem{lem}[thm]{Lemma}
\newtheorem{defn}[thm]{Definition}
\newtheorem{prop}[thm]{Proposition}
\newtheorem{cor}[thm]{Corollary}
\newtheorem{rem}[thm]{Remark}

\newcommand{\C}{{\mathbb C}}
\newcommand{\R}{{\mathbb R}}
\newcommand{\Z}{{\mathbb Z}}
\newcommand{\N}{{\mathbb N}}
\newcommand{\B}{{\mathbb B}}
\newcommand{\bB}{{\mathbb S}}
\newcommand{\D}{{\mathbb D}}

\newcommand{\I}{{\mathcal I}}
\newcommand{\T}{{\mathcal T}}
\newcommand{\K}{{\mathcal K}}

\newcommand{\dbar}{\bar\partial}
\newcommand{\e}{\varepsilon}
\newcommand{\p}{\partial}
\newcommand{\pphi}{\widetilde{\varphi}}
\newcommand{\ppsi}{\widetilde{\psi}}

\newcommand{\z}{\zeta}

\newcommand{\blankline}{\hspace*{\fill}\mbox{ }}

\title[Equations related to Toeplitz operators]{On integral equations related to \\ weighted Toeplitz operators}
\author{Carme Cascante}
\address{C. Cascante: Dept.\ Matem\`atica Aplicada i An\`alisi,Universitat  de Barcelona, Gran Via 585, 08071 Barcelona, Spain}
\email{cascante@ub.edu}
\author{Joan F\`abrega}
\address{Joan F\`abrega: Dept.\ Matem\`atica Aplicada i An\`alisi,Universitat  de Barcelona, Gran Via 585, 08071 Barcelona, Spain}
\email{joan$_{-}$fabrega@ub.edu}
\author{Daniel Pascuas}
\address{D. Pascuas: Dept.\ Matem\`atica Aplicada i An\`alisi,Universitat  de Barcelona, Gran Via 585, 08071 Barcelona, Spain}
\email{daniel$_{-}$pascuas@ub.edu}

\date{\today}
\thanks{Partially supported by  DGICYT Grant MTM2008-05561-C02-01, DURSI Grant 2009SGR 1303 and Grant MTM2007-30904-E}

\begin{document}

\begin{abstract} 
For weighted Toeplitz operators $\T^N_\varphi$ defined on spaces of holomorphic functions in the unit ball,
 we derive regularity properties of the solutions $f$ to the integral equation $\T^N_\varphi(f)=h$ in terms of the
 regularity of the symbol  $\varphi$ and the data $h$. As an application, 
we deduce that if $f\not\equiv0$ is a function in the Hardy space $H^1$ such that its argument  $\bar f/f$ is in a Lipschitz space on the unit sphere $\bB$, then $f$ is also in the same Lipschitz space, extending a result 
of K. Dyakonov to several complex variables.
% We also deduce that if $f$ is a function in the ball algebra with no zeros on $\bB$, satisfying that $|f|$ is in
% a Lipschitz space on $\bB$, then $f$ is also in the same space.
\end{abstract}

\maketitle

\section{Introduction}
The goal of this paper is to study the regularity of solutions to certain equations related to weighted Toeplitz operators in several complex variables.

We will start by stating some particular cases of the main results in this paper, which involve classical spaces and integral operators and illustrate the object of this paper, although they can be applied in a more general setting.

Let $\B$ denote the open unit ball in $\C^n$ and $\bB$ its boundary.
In the one variable setting ($n=1$), $\B$ and $\bB$ will also be denoted by $\D$ and $\mathbb{T}$, respectively. For any $\tau>0$, $\Lambda_\tau=\Lambda_\tau(\bB)$ is the classical  Lipschitz-Zygmund space on $\bB$.
%We will denote by $\Gamma_\tau=\Gamma_\tau,\,\tau>0$ the Zygmund-Lipschitz space on $\bB$ with respect to the pseudodistance $d(\z,\eta)=|1-\bar\z\eta|$ (see Subsection \ref{sec:G} for precise definitions). Since $|1-\bar\z\eta|\le|\z-\eta|$, it is clear that $\Gamma_\tau$ is a subspace of the Euclidean Zygmund-Lipschitz space $\Lambda_\tau(\bB)$.

If $\varphi\in \Lambda_\tau$, we consider the Toeplitz operator 
$\T_\varphi:H^1\rightarrow H^1$, defined by 
$\T_\varphi(f)(z):=\mathcal{P}(\varphi f)(z)$, where $\mathcal{P}$ is the Cauchy projection, given by
$$
\mathcal{P}(\psi)(z):=\int_{\bB} \frac{\psi(\z)}{(1-\bar \z z)^{n}}d\sigma(\z)\qquad(\psi\in L^1(\bB)).
$$
Here $d\sigma$ denotes the normalized Lebesgue measure on $\bB$.
We point out that  $\T_\varphi$ maps $H^1$ to itself because $\varphi\in \Lambda_\tau$.  

For this scale of Lispchitz spaces we prove the following result:

\begin{thm} \label{thm:IntrLip} Let $\tau>0$ and $\varphi\in \Lambda_\tau$ be a non-vanishing function on $\bB$. If $f\in H^1$ and $\T_\varphi(f)\in \Lambda_\tau$, then $f\in\Lambda_\tau$.
\end{thm}

This result extends \cite[Theorem 3.1]{KD}, which deals with the case  $n=1$ and the regularity of the solutions to the equation $\T_\varphi(f)=0$. 

We remark that if we drop the condition $0\notin\varphi(\bB)$, then Theorem \ref{thm:IntrLip} is not true in general. Indeed, we only need to consider the symbol $\varphi(\z)=(1-\z_1)^\tau$ and the function $f(\z)=(1-\z_1)^{-\tau}$ with  $0<\tau<n$ (to ensure that $f\in H^1$).

As in the one variable case (see \cite{KD}), the above theorem implies some interesting properties of the 
holomorphic Lipschitz functions.
% this result we obtain some results concerning the regularity on the terms corresponding to some decompositions. 
For instance, 

\begin{cor}\label{cor:Intarg0}
If $f\in H^1$, $\varphi\in \Lambda_\tau$, such that $0\notin \varphi(\bB)$ and 
$\varphi f\in \Lambda_\tau+\ker \mathcal{P}\subset L^1(\bB)$, then $f\in\Lambda_\tau$. 
\end{cor}

In particular, we have: 
\begin{cor}\label{cor:Intarg}
If $f\in H^1\setminus\{0\}$ and its argument function $\varphi=\bar f/f$ is in $\Lambda_\tau$,
 then $f\in\Lambda_\tau$. 
\end{cor}

The preceding corollary is proved in~{\cite{KD}} for $n=1$.

 %The main purpose of this paper is to extend these results to weighted Toeplitz operators associated to
 % more general symbols.

In this paper we prove the above results and extend them to weighted Toeplitz operators associated to
 more general symbols.

We will denote by $\Gamma_\tau=\Gamma_\tau(\bB),\,\tau>0$, the Lipschitz-Zygmund space on $\bB$ with respect to the pseudodistance $d(\z,\eta)=|1-\bar\z\eta|$ (see Subsection \ref{sec:G} for precise definitions).
 Since $|1-\bar\z\eta|\le|\z-\eta|$, it is clear that $\Gamma_\tau$ is a subspace of $\Lambda_\tau$. For a positive integer $k$, and real numbers $0<\tau_0\le \tau<k$, $0<\tau_0< 1/2$, we will consider  spaces $G^{\tau_0}_{\tau,k}(\B)\subset \Lambda_{\tau_0}(\bar \B)\cap \mathcal{C}^k(\B)$, whose restrictions to $\bB$ contain the space $\Gamma_\tau$, and also the space  $\Lambda_\tau$, for $\tau>1/2$. Moreover, they satisfy that their intersection with the space $H=H(\B)$ of holomorphic functions on $\B$, coincides with the Lipschitz-Zygmund space  of holomorphic functions on $\B$, denoted by $B^\infty_\tau$. These holomorphic spaces $B^\infty_\tau$ are characterized in terms of the growth of the derivatives  (see Subsection \ref{sec:G} for a precise definition and their main properties). 

For $N>0$, let $d\nu_N(z):=c_N(1-|z|^2)^{N-1}d\nu(z)$, where $\nu$ is the Lebesgue measure on $\B$
and $c_N=\frac{\Gamma(n+N)}{n!\Gamma(N)}$, so that $\nu_N(\B)=1$. Let $L^1_N:=L^1(B,d\nu_N)$ and consider the weighted Bergman
 projection $\mathcal{P}^N:L^1_N\to H$  defined by
$$
\mathcal{P}^N(\psi)(z):=\int_{\B} \frac{\psi(w)}{(1-\bar w z)^{n+N}}\,d\nu_N(w).
$$
 Let $B^1_{-N}:=L^1_N\cap H$ and for $\varphi\in L^{\infty}(\B)$, define the weighted Toeplitz operator
 $\T^N_\varphi:B^1_{-N}\rightarrow H$ by $\T^N_\varphi(f):=\mathcal{P}^N(\varphi f)$.
Since $\lim_{N\searrow 0} \mathcal{P}^N(\psi)=\mathcal{P}(\psi)$ (see \cite[\S\,0.3]{Bea-Bur}), we extend these definitions
 to $N=0$, by $\mathcal{P}^0=\mathcal{P}$ and $\T_\varphi^0=\T_\varphi$, $\varphi\in L^\infty(\bB)$.
 In these cases, the operators are defined on $L^1(\bB)$ and $H^1$, respectively.

The next two theorems are the main results of this paper.

\begin{thm}\label{thm:kerNInt}  Let  $\varphi\in G^{\tau_0}_{\tau,k}$ such that $0\notin\varphi(\bB)$.
 If $f\in B^1_{-N}$ satisfies $\T^N_{\varphi}(f)=h\in B^\infty_\tau$, then $f\in B^\infty_\tau$ and 
$\|f\|_{B_\tau^\infty}\le C\left(\|f\|_{B^1_{-N}}+\|h\|_{B^\infty_\tau}\right)$, where 
$C>0$ is a finite constant only depending on $\varphi$, $N>0$ and $n$.
 In particular, $\|f\|_{B_\tau^\infty}\leq C\|f\|_{B_{-N}^1}$, for  any $f\in\ker \T^N_{\varphi}$.
\end{thm} 

The corresponding statement for the case $N=0$ is:

\begin{thm}\label{thm:ker0Int} Let $\varphi$ be the restriction to $\bB$ of a function in $G^{\tau_0}_{\tau,k}$ and
 let $f\in H^1$. If $0\notin \varphi(\bB)$ and $\T_\varphi (f)=h\in B^\infty_\tau$, then $f\in B^\infty_\tau$
 and $\|f\|_{B_\tau^\infty}\le C\left(\|f\|_{H^1}+\|h\|_{B^\infty_\tau}\right)$,  where 
$C>0$ is a finite constant only depending on $\varphi$ and $n$.
 In particular, $\|f\|_{B_\tau^\infty}\leq C\|f\|_{H^1}$, for  any $f\in\ker \T_{\varphi}$.
\end{thm} 

The preceding theorem was proved in \cite[Theorem 3.1]{KD} for $n=1$, $\varphi\in\Lambda_\tau$ and
  $h=0$.

Note that the inequalities in the above theorems are in fact equivalences due to the continuity of
 both the Toeplitz operator and the embeddings $B^\infty_\tau\subset H^1\subset B^1_{-N}$.

% In fact, it can be obtained quantitative estimates of the norm of the solutions of the Toeplitz equation
% $\T^N_{\varphi}(f)=h$. Namely, the methods of the proof actually shows that there exists a constant
% $C_\varphi$ depending on $\varphi$, such that
% $\|f\|_{B_\tau^\infty}\le C_\varphi\left(\|f\|_{B^1_{-N}}+\|h\|_{B^\infty_\tau}\right)$  if $N>0$
% or $\|f\|_{B_\tau^\infty}\le C_\varphi\left(\|f\|_{H^1}+\|h\|_{B^\infty_\tau}\right)$ if $N=0$.
% In the particular case that $h=0$, then there exists a constant $C_\varphi$ such that
% $\|f\|_{B_\tau^\infty}\leq C_\varphi\|f\|_{B_{-N}^1}$ (respectively $\|f\|_{H^1}$ if $N=0$) for any $f$
% in the kernel of the operator $\T^N_{\varphi}$. 
% This last estimate for $n=1$ and $N=0$ was proved in \cite[Theorem 3.1]{KD}. 

For $\tau>1/2$,  the restriction to $\bB$ of $G_{\tau,k}^{\tau_0}$ contains the space  $\Lambda_\tau$,  hence Theorem~{\ref{thm:ker0Int}} includes the result of Theorem \ref{thm:IntrLip} for these cases. 
However, the same techniques used to prove the above theorems allow us to extend this result to the whole scale of spaces $\Lambda_\tau$.

\begin{cor} If either $f\in B^1_{-N}$, for $N>0$, or $f\in H^1$ and $N=0$, $\varphi\in G_{\tau,k}^{\tau_0}$, $0\notin\varphi(\bB)$ and $\varphi f\in G_{\tau,k}^{\tau_0}+\ker \mathcal{P}^N$,  then $f\in B^\infty_\tau$.
\end{cor}

In particular if $g\in H^1$, then $\overline{g-g(0)}\in\ker\mathcal{P}$, and therefore we have:

\begin{cor} 
If $f,g\in H^1$ satisfy $\varphi=\bar g/f\in \Gamma_{\tau}$ and $0\notin\varphi(\bB)$,
 then $f,g\in B^\infty_\tau$. 
\end{cor}

 The preceding result generalizes Corollary~{\ref{cor:Intarg},} and extends~\cite[Corollary 3.2]{KD}
 to dimension $n>1$.

% Theorem \ref{thm:IntrLip} can be rewritten in order to obtain properties related to the spectrum of Toeplitz operators with Zygmund-Lipschitz symbols. We have:
% 
%\begin{cor} 
%If $\varphi\in\Lambda_\tau$ and $\lambda\notin \varphi(\bB)$, then  the solutions $f$ of the equation $\T_\varphi(f)=\lambda f$ are in $\Lambda_\tau$.
%\end{cor}

The paper is organized as follows. 
In Section~2 we  state some  properties of  spaces considered in this paper and we also recall some integral
representation formulas used in the proof of the main theorems. 

 In Section~3 we state our main technical theorem (Theorem~{\ref{thm:kerT1}}) from which we deduce 
Theorems~{\ref{thm:kerNInt}} and~{\ref{thm:ker0Int}} and its corollaries. 
We also construct some counterexamples.  Finally,
 Theorem~{\ref{thm:kerT1}} is proved in Section~4.

%\textbf{Notations:} Throughout the paper, the letter $C$ will denote a positive constant, which may vary from place to
% place. The notation $f(z)\lesssim g(z)$ means that there exists $C>0$, which does not depend on $z$, $f$ and $g$,
% such that $f(z)\le C g(z)$. 

\section{Preliminaries}

%{\bf 
%In this section we state some well-known facts on  spaces of functions and integral operators that
% will be used  in the forthcoming sections.
%}

\subsection{Notations}
Throughout the paper, the letter $C$ will denote a positive constant, which may vary from place to
 place. The notation $f(z)\lesssim g(z)$ means that there exists $C>0$, which does not depend on $z$, $f$ and $g$,
 such that $f(z)\le C g(z)$. We write $f(z)\approx g(z)$ when $f(z)\lesssim g(z)$ and $g(z)\lesssim f(z)$.

Let $\p_j:=\frac{\p\,}{\p z_j}$, for $j=1,\dots,n$.
 For any multiindex $\alpha$, i.e. $\alpha=(\alpha_1,\cdots,\alpha_n)\in\N^n$,
 where $\N$ is the set of non-negative integers, let $|\alpha|:=\sum_{j=1}^{n}\alpha_j$ and 
 $\p_\alpha:=\frac{\p^{|\alpha|}}{\p z^\alpha}=\p_1^{\alpha_1}\cdots\p_n^{\alpha_n}$.
 We write   $|\p^k\varphi|:=\sum_{|\alpha|=k}|\p_\alpha\varphi|$ and 
            $|d^k\varphi|:=\sum_{|\alpha|+|\beta|=k}|\p_\alpha\bar\p_\beta\varphi|$.
When $n>1$, we also consider the complex tangential differential operators
 $\mathcal{D}_{i,j}:=\bar z_i\p_j-\bar z_j\p_i$ and
 $\displaystyle{|\p_T\varphi|:=\sum_{1\le i<j\le n}|\mathcal{D}_{i,j}\varphi|}$.
 For $n=1$, we write $|\p_T\varphi|:=0$. 

We also introduce the following two functions which  will be used in the definition of the spaces $G^{\tau_0}_{\tau,k}$.

For $\varphi\in \mathcal{C}^1(\B)$ , let 
\begin{equation}\label{eqn:def:pphi}
\pphi(z):=(1-|z|^2)|\bar\p \varphi(z)|+(1-|z|^2)^{1/2}|\bar\p_T \varphi(z)|.
\end{equation}

For $t\in\R$, let $\omega_t$ be the function on $\B$ defined by
\begin{equation} \label{eqn:omega}
\omega_t(z):=\left\{
\begin{array}{ll}
  (1-|z|^2)^{\min(t,0)},& \mbox{if $t\ne0$,}\\                    
  \log\frac{e}{1-|z|^2},& \mbox{if $t=0$.}
\end{array}\right.
\end{equation}

%To conclude we recall that if $N>0$, then $d\nu_N(z)=\dfrac{\Gamma(n+N)}{n!\Gamma(N)}(1-|z|^2)^{N-1}d\nu(z)$, and that we extend this definition to $N=0$ denoting $d\sigma=d\nu_0$. In order to unify the notations and the assertions in some results, we also extend the definition of $L^p_N=L^p(d\nu_N), N>0$, denoting  $L^p_0=L^p(d\sigma)$ and $\displaystyle{\int_{\B}\varphi d\nu_0=\int_{\bB}\varphi d\sigma}$.

\subsection{Holomorphic Besov spaces}

Let $1\le p<\infty$, $k\in\N$ and $\delta>0$. The weighted Sobolev space $L^p_{k,\delta}$ is the completion of the space $\mathcal{C}^\infty(\bar \B)$, endowed with the norm
$$
\|\psi\|_{L^p_{k,\delta}}:=\Big\{\sum_{j=0}^k
             \int_{\B}|d^j \psi(z)|^p(1-|z|^2)^{\delta p-1}d\nu(z)\Big\}^{1/p}.
$$
When $k=0$, we will just write $L_\delta^p=L^p_{0,\delta}$.
We extend this definition to the case $p=\infty$, so that
 $L^\infty_{k,\delta}$ is the subspace of functions $\psi$ in the Sobolev space $L^1_{k,\delta+1}$
 satisfying 
$$
\|\psi\|_{L^\infty_{k,\delta}}:=\sum_{j=0}^k\sup_{z\in \B}
                             |d^j \psi(z)|(1-|z|^2)^{\delta}<\infty.
$$

If $1\le p\le\infty$ and $s\in\R$, the holomorphic Besov space $B^p_s$ is defined to be
 $B^p_s:=H\cap L^p_{k,k-s}$, for some $k\in\N$, $k>s$.
 It is well known that  $\|\cdot\|_{L^p_{m,m-s}}$ and $\|\cdot\|_{L^p_{k,k-s}}$
 are equivalent norms on $B^p_s$, for any $m,k\in\N$, $m,k>s$.

Note that if $s=0$, then $B^\infty_0$ is the Bloch space and if $s>0$ then $B^\infty_s$ coincides with the space of holomorphic functions on $\B$ whose boundary values are in the corresponding Lipschitz-Zygmund space $\Lambda_s$
 (see the next subsection for the precise definitions of these last two spaces).

\begin{prop}[{\cite[Theorems 5.13,14]{Bea-Bur}}]    \label{prop:embed}
 Let $1\le p\le q\le \infty$ and let $s,t\in\R$. Then:
\begin{enumerate}
	\item \label{item:embed2} If $s>t$, then $B^p_s\subset B^p_t$.
	\item For any $\e>0$, $B^1_0\subset H^1\subset B^1_{-\e}$.
	\item \label{item:embed3} If $s-n/p=t-n/q$, then $B^1_{s+n/p'}\subset B^p_s\subset B^q_t\subset B^\infty_{s-n/p}$.
\end{enumerate}
\end{prop}

\subsection{The space $G_{\tau,k}^{\tau_0}$}\label{sec:G}
 In this section we define the spaces $G_{\tau,k}^{\tau_0}$ and we state some of their main properties.

%For $t\in\R$, let $\omega_t$ be the function on $\B$ defined by
%\begin{equation} \label{eqn:omega}
%\omega_t(z):=\left\{
%\begin{array}{ll}
%  (1-|z|^2)^{\min(t,0)},& \mbox{if $t\ne0$,}\\                    
%  \log\frac{e}{1-|z|^2},& \mbox{if $t=0$.}
%\end{array}\right.
%\end{equation}

\begin{defn} \label{def:G}
Let $0<\tau_0\le \tau$, $\tau_0< 1/2$, and let $k>\tau$ be an integer. The space $G^{\tau_0}_{\tau,k}$ consists of all functions $\varphi\in \mathcal{C}^{k}(\B)\cap \mathcal{C}(\bar \B)$ satisfying
$$
\|\varphi\|_{G^{\tau_0}_{\tau,k}} :=\sum_{j=0}^k\sup_{z\in \B} 
	\frac{|\p^j \varphi(z)|}{\omega_{\tau-j}(z)}+\sup_{z\in \B}\frac{\pphi(z)}{(1-|z|^2)^{\tau_0}}<\infty,
$$
\end{defn}
 
%\begin{equation}\label{eqn:def:pphi}
%\pphi(z):=(1-|z|^2)|\bar\p \varphi(z)|+(1-|z|^2)^{1/2}|\bar\p_T \varphi(z)|.
%\end{equation}
%\end{defn}

 Note that $H\cap G^{\tau_0}_{\tau,k}=B^\infty_\tau$, and $G^{\tau_0}_{\tau,k}\cdot L^p_\delta\subset L^p_\delta$, since $G^{\tau_0}_{\tau,k}\subset L^{\infty}(\B)$.

The following embedding is a consequence of the definition of $G^{\tau_0}_{\tau,k}$ and the fact that  $(1-|z|^2)^s\le (1-|z|^2)^t$ and $\omega_s(z)\lesssim \omega_t(z)$, if $s>t$.

\begin{lem} \label{lem:embedG}  
$G^{\tau_0}_{\tau,k}\subset G^{\vartheta_0}_{\vartheta,m}$, provided that $\vartheta_0\le\tau_0$,
 $\vartheta\le\tau$ and $m\le k$.
\end{lem}

In order to obtain multiplicative properties of the spaces $G^{\tau_0}_{\tau,k}$,   we first state some properties of the function $\omega_t$.

\begin{lem}\label{lem:poteomegas}
Let $a,b\in\R$. Then 
$$
(1-|z|^2)^a\,\omega_b(z) \lesssim (1-|z|^2)^c,
$$
for every $c\in\R$ such that $c<a$ and $c\le a+b$.
\end{lem}

\begin{proof}
 Just note that
$$
(1-|z|^2)^a\,\omega_b(z) =
\left\{
\begin{array}{ll}
(1-|z|^2)^{\min(a+b,a)},           & \mbox{if $b\ne 0$} \\
(1-|z|^2)^a \log\frac{e}{1-|z|^2}, & \mbox{if $b= 0$}
\end{array}
\right\} \lesssim (1-|z|^2)^c,
$$
for every $c\in\R$ such that $c<a$ and $c\le a+b$.
\end{proof}

\begin{lem}\label{lem:prodomegas} 
Let $\vartheta,\tau>0$, $k\in\R$ and $m\in\Z$ such that $m\ge0$. Then
$$
S^{\vartheta,\tau}_{m,k}:=
\sum_{i=0}^{m} \omega_{\vartheta-i}\,\omega_{\tau+i-k}
\lesssim \omega_{\vartheta-m}+\omega_{\tau-k}.
$$
\end{lem}
\begin{proof}
We estimate the different products $\omega_{\vartheta-i}\,\omega_{\tau+i-k}$ as follows:\vspace*{5pt}

\noindent
$\bullet$ If $i>k-\tau$, then $\omega_{\vartheta-i}\,\omega_{\tau+i-k}=\omega_{\vartheta-i}\lesssim \omega_{\vartheta-m}$,
since $\vartheta-m\le\vartheta-i$.
\vspace*{5pt}

\noindent
$\bullet$ If $i<\vartheta$, then $\omega_{\vartheta-i}\,\omega_{\tau+i-k}=\omega_{\tau+i-k}\lesssim \omega_{\tau-k}$,
since $\tau-k\le\tau+i-k$.
\vspace*{5pt}

\noindent
$\bullet$ If $i>\vartheta$ and $i\le k-\tau$, then
$$
\omega_{\vartheta-i}(z)\omega_{\tau+i-k}(z)
=(1-|z|)^{\vartheta-i}\omega_{\tau+i-k}(z)\lesssim (1-|z|)^{\tau-k}=\omega_{\tau-k}(z),
$$
 by Lemma~{\ref{lem:poteomegas}}, since $\tau-k<\tau-k+\vartheta=(\vartheta-i)+(\tau+i-k)$
 and $\tau-k\le -i<-\vartheta<0$.
\vspace*{5pt}

\noindent
$\bullet$ If $i=\vartheta$ and $i<k-\tau$, then
$$
\omega_{\vartheta-i}(z)\omega_{\tau+i-k}(z)=(1-|z|)^{\tau+i-k}\omega_0(z)\lesssim (1-|z|)^{\tau-k}=\omega_{\tau-k}(z),
$$
 by Lemma~{\ref{lem:poteomegas}}, since $\tau-k<\tau-k+\vartheta=\tau+i-k<0$.
\vspace*{5pt}

\noindent
$\bullet$  If $\vartheta=i=k-\tau$, then
$$
\omega_{\vartheta-i}(z)\omega_{\tau+i-k}(z)=\omega_0(z)^2\lesssim (1-|z|)^{\tau-k}=\omega_{\tau-k}(z),
$$
since $\tau-k=-\vartheta<0$.
\end{proof}

\begin{prop} \label{prop:prodGs}
$$
\|\varphi\psi\|_{G^{\vartheta_0}_{\vartheta,m}}\lesssim
\|\varphi\|_{G^{\tau_0}_{\tau,k}}\|\psi\|_{G^{\vartheta_0}_{\vartheta,m}}
\qquad(\varphi\in G^{\tau_0}_{\tau,k},\,\psi\in G^{\tau_0}_{\vartheta,m}),
$$
provided that  $\vartheta_0\le\tau_0$, $\vartheta\le\tau$ and $m\le k$.
\end{prop}

\begin{proof} If $\alpha\in\N^n$, $|\alpha|=l\le m$, then
\begin{eqnarray*}
|\p_\alpha(\varphi\psi)|
&\lesssim& \sum_{\beta+\gamma=\alpha}|\p_\beta\varphi||\p_\gamma\psi|
\lesssim \|\varphi\|_{G^{\tau_0}_{\tau,k}}\|\psi\|_{G^{\vartheta_0}_{\vartheta,m}}S^{\vartheta,\tau}_{l,l}\\
&\lesssim& \|\varphi\|_{G^{\tau_0}_{\tau,k}}\|\psi\|_{G^{\vartheta_0}_{\vartheta,m}}\omega_{\vartheta-l},
\end{eqnarray*}
since, by Lemma~{\ref{lem:prodomegas}},
 $S^{\vartheta,\tau}_{l,l}\lesssim\omega_{\vartheta-l}+\omega_{\tau-l}\lesssim\omega_{\vartheta-l}$.
 
On the other hand,
$$
\widetilde{\varphi\psi}(z)\le|\varphi(z)|\widetilde{\psi}(z)+\widetilde{\varphi}(z)|\psi(z)|
\lesssim \|\varphi\|_{G^{\tau_0}_{\tau,k}}\|\psi\|_{G^{\vartheta_0}_{\vartheta,m}}(1-|z|^2)^{\vartheta_0},
$$
and the proof is complete.
\end{proof}
 
Our next goal is to show the connection between the spaces $G^{\tau_0}_{\tau,k}$ and both
 the non-isotropic Lipschitz-Zygmund spaces $\Gamma_\tau$ and the classical Lipschitz-Zygmund
 spaces $\Lambda_\tau$.

If $0<\tau<1$, the classical Lipschitz-Zygmund space  on $\bB$,  $\Lambda_\tau=\Lambda_{\tau}(\bB)$,
 with respect to the Euclidean metric consists of all the functions $\varphi \in\mathcal{C}(\bB)$ such that
$$
\|\varphi\|_{\Lambda_\tau}:=\|\varphi\|_{\infty}+\sup_{\substack{\z,\eta\in S\\ \z\ne\eta}} \frac{|\varphi(\z)-\varphi(\eta)|}{|\z-\eta|^\tau}<\infty.
$$
If $k$ is a positive integer and $k<\tau<k+1$, then $\Lambda_\tau=\Lambda_{\tau}(\bB)$ consists of all the
 functions $\varphi\in \mathcal{C}^k(\bB)$ such that
$$
\|\varphi\|_{\Lambda_\tau}:=\|\varphi\|_{\mathcal{C}^k}+\sum_{|\alpha|+|\beta|=k} \|\p_\alpha \bar\p_\beta\varphi\|_{\Lambda_{\tau-k}(\bB)}<\infty.
$$

When $\tau$ is a positive integer, $\Lambda_\tau$ is defined analogously by using second order differences. The spaces $\Lambda_\tau(\B)$ are defined in a similar way.
 
The main properties of the spaces $\Lambda_\tau$ can be found, for instance, in the expository paper~{\cite{Kr1}.} 

It is well known (see~{\cite[\S\,15]{Kr1})} that a continuous function 
$\varphi$ is in $\Lambda_\tau$ if and only if, for some (any) integer $k>\tau$, its harmonic extension $\Phi$ on $\B$ satisfies 
\begin{equation}\label{eqn:lipder}
\sup_{z\in \B}(1-|z|^2)^{k-\tau}|d^k\Phi(z)| <\infty.
\end{equation}

We recall that if~{\eqref{eqn:lipder}} holds for some function $\varphi\in\mathcal{C}^k(\B)$,
 then $\varphi\in\Lambda_\tau(\B)$ (see~{\cite[Theorem 15.7]{Kr1}}).
 
This fact and the estimate $|d\varphi(z)|\lesssim \|\varphi\|_{G^{\tau_0}_{\tau,k}}(1-|z|^2)^{\tau_0-1}$ give

\begin{prop} \label{prop:GLip0} $G^{\tau_0}_{\tau,k}\subset \Lambda_{\tau_0}(\B)$.
\end{prop} 

We also consider the Lipschitz-Zygmund space on $\bB$ with respect to the pseudodistance $d(\z,\eta)=|1-\bar\z\eta|$,
which is denoted by $\Gamma_{\tau}(\bB)$. If $0<\tau<1/2$, this space is defined just as $\Lambda_\tau$ but replacing the Euclidean distance $|\zeta-\eta|$ by $d(\z,\eta)$. For values $\tau\ge 1/2$ the definition is given in terms of Lipschitz conditions of certain complex tangential derivatives (see \cite[pp.\,670-1]{Bo-Br-Gr} and the references therein  for the precise definitions and main properties). 

%Note that $\Gamma_\tau=\Lambda_\tau$, when $n=1$, while, for $n>1$,  $\Gamma_\tau$ is the space of functions
% $\varphi\in \Lambda_\tau$ such that, for any complex tangential curve $\gamma:[0,1]\rightarrow S$, satisfy
% $\varphi\circ\gamma\in\Lambda_{2\tau}([0,1])$.

We recall that if $f\in H(\B)$ has boundary values $f^*$, then $f^*$ in $\Lambda_\tau$,
 if and only if  $f^*\in\Gamma_\tau$  (see \cite{Ste} or \cite[\S 6.4]{Ru} or \cite[\S 8.8]{Kr3} and the references therein. See also \cite[pp.\,670-1]{Bo-Br-Gr}).
 If $0<\tau<n$, the functions in  $\Gamma_\tau$ can be described in terms of their
 invariant harmonic extensions. In this case, we have that $\varphi$ is in 
$\Gamma_\tau$ if and only if, for some (any) integer $k>\tau$, its invariant harmonic extension $\Phi$ on $\B$ satisfies~{\eqref{eqn:lipder}}.
 This characterization fails to be true when $\tau\ge n$ 
(see~{\cite[Chapter 6]{Kr2}} for more details). Similarly to what happens in the holomorphic case, the complex tangential derivatives of the functions in the  space $\Gamma_\tau$ are more regular, in the sense that $\mathcal{D}_{ij}\varphi\in \Gamma_{\tau-1/2}$ for $i,j=1,\dots,n$.

The next results relate the spaces $\Lambda_\tau$ and $\Gamma_\tau$ to $G_{\tau,k}^{\tau_0}$. 

\begin{prop}\label{prop:GLip1}\blankline

\noindent
a) If $n=1$ then the harmonic extension of a function in $\Lambda_\tau$ belongs to any space $G_{\tau,k}^{\tau_0}$.
% for every integer $.
 %>\tau$ and every $0<\tau_0\le\min\{\tau,1/2\}$.

\noindent
b) If $n>1$ and  $\tau>1/2$ then every $\varphi\in\Lambda_\tau$ is the restriction of a function
 $\Phi\in G_{\tau,k}^{\tau_0}$. Namely, 
for any integer $k>\tau$, the harmonic extension $\Phi$ of $\varphi$ satisfies that:  
\begin{itemize}
\item $\Phi\in G_{\tau,k}^{\tau-1/2}$, when $1/2<\tau<1$.
\item $\Phi\in G_{\tau,k}^{\tau_0}$, for any $0<\tau_0<1/2$, when $\tau\ge 1$.
%\item $\Phi\in G_{\tau,k}^{1/2}$, when $1<\tau$.  
\end{itemize}
\end{prop}

\begin{cor} \label{cor:Glip1} If either $n=1$ or $n>1$ and $\tau>1/2$, then 
 every $\varphi\in\Lambda_\tau$ is the restriction of a function
 $\Phi\in G_{\tau,k}^{\tau_0}$, for some $0<\tau_0$ and for any integer $k$.
\end{cor}

\begin{prop} \label{prop:GLip2} 
 If $n>1$, $0<\tau<n$ and $\varphi\in\Gamma_\tau$, then, for any integer $k>\tau$, its invariant
 harmonic extension $\Phi$ satisfies that:
\begin{itemize}
\item $\Phi\in G^{\tau}_{\tau,k}$, when $0<\tau<1/2$. 
\item $\Phi\in G^{\tau_0}_{1,k}$, for any $0<\tau_0<1/2$, when $\tau\ge 1/2$.
%\item $\Phi\in G^{1/2}_{\tau,k}$, when $1/2<\tau<n$.
\end{itemize}
\end{prop}

\begin{cor} \label{cor:Glip2} Every $\varphi\in \Gamma_\tau$, $\tau>0$, is the restriction of a function
 $\Phi\in G_{\tau,k}^{\tau_0}$, for some $0<\tau_0$ and for any integer $k$.
\end{cor}

\subsection{Representation formulas and estimates}

In this subsection we recall some well-known results on the integral representation formulas obtained in \cite{Char}.

We begin by introducing the following nonnegative integral kernels and their corresponding integral operators. 
\begin{defn} 
Let $N,M,L\in\R$ such that $N>0$ and $L<n$. Then
$$
\mathcal{K}^{N}_{M,L}(w,z):=\frac{(1-|w|^2)^{N-1}}{|1-\bar w z|^{M}D(w,z)^L},
\quad(z, w\in \bar\B,\,\,z\ne w), 
$$ 
where $D(w,z):=|1-\bar w z|^2-(1-|w|^2)(1-|z|^2)$. 
 The associated integral operator is also denoted by $\mathcal{K}^{N}_{M,L}$: 
$$
\mathcal{K}^{N}_{M,L}(\psi)(z):=\int_{\B} \mathcal{K}^{N}_{M,L}(w,z)\psi(w)\,d\nu(w).
$$ 
\end{defn}

Note that $D(w,z)=|(w-z)\bar z|^2+(1-|z|^2)|w-z|^2$, so, for every $z\in \B$ such that $1-|z|^2\ge \delta>0$, we have that 
\begin{equation}\label{estimate:kernel:K}
\mathcal{K}^{N}_{M,L}(w,z)\simeq
\left\{\begin{array}{ll}
|w-z|^{-2L},     & \mbox{ if $|w-z|<(1-|z|)/2$,}\\
(1-|w|^2)^{N-1}, & \mbox{ if $|w-z|\ge(1-|z|)/2$.}
\end{array}\right. 
\end{equation}

\begin{thm}[{\cite{Char}}]\label{thm:rep}
 Let $N>0$. Then every function $\psi\in\mathcal{C}^1(\bar \B)$ decomposes as 
\begin{equation} \label{eqn:rep} 
\psi= \mathcal{P}^N(\psi)+\mathcal{K}^N(\bar\p \psi),
\end{equation}
where 
$$
\mathcal{K}^N(\bar\p \psi)(z):=\int_{\B}\mathcal{K}^N(w,z)\wedge\bar\p\psi(w)
$$
and $\mathcal{K}^N(w,z)$ is an $(n,n-1)$-form (on $w$) of class 
$\mathcal{C}^{\infty}$ on $\B\times \B$ outside its diagonal.

In particular if $\psi$ is holomorphic on $\B$ then $\psi=\mathcal{P}^N(\psi)$.

Moreover, $\mathcal{K}^N(w,z)$ satisfies the estimate
\begin{equation}\label{eqn:estimate:kernel:K}
|\mathcal{K}^N(w,z)\wedge\bar\p\psi(w)|\lesssim K^N_{N-n+1,n-1/2}(w,z) \ppsi(w),
\end{equation}
for any $\psi\in\mathcal{C}^1(\B)$, where $\ppsi$ is defined as in~{\eqref{eqn:def:pphi}}. 
\end{thm}
Then, it is clear that, 
\begin{equation}\label{eqn:estimates:P:K}
|\mathcal{\mathcal{P}}^N(\psi)|\lesssim \K^N_{n+N,0}(|\psi|) \quad\mbox{  and  }\quad |\mathcal{\K}^N(\bar\p\psi)|\lesssim \K^N_{N-n+1,n-1/2}(\ppsi).
\end{equation}
  
\begin{rem}\label{rem:thm:rep}
 The above representation formula will be applied in a more general setting to functions $\psi=\varphi f$
 where $\varphi\in G^{\tau_0}_{\tau,k}$ and either $f\in B^1_{-N}$, for $N>0$, or  $f\in H^1$, for $N=0$. The validity 
 of the formula for this class of functions is obtained by applying the dominated convergence theorem and
 Theorem~{\ref{thm:rep}} to the functions $\psi_r(z)=\psi(rz)$.
\end{rem}

\begin{lem}[{\cite[Lemma I.1]{Char}}] \label{lem:estK}  
$$
\int_{\B} \mathcal{K}^{N}_{M,L}(w,z)d\nu(w)\lesssim\omega_t(z),
$$
where $t:=n+N-M-2L$ is the so-called {\sf type} of the kernel $\mathcal{K}^{N}_{M,L}$. 
\end{lem}

Observe that from the above estimate we deduce that if $\K^N_{M,L}$ is a kernel of type 0, $0<\delta<N$ and $\psi(z)=(1-|z|^2)^{-\delta}$, then $\K^N_{M,L}(\psi)\lesssim \psi$.
As a consequence of that result and Schur's lemma we have:

\begin{lem} \label{lem:estK:Lp}  
If $\K^N_{M,L}$ is a kernel of type 0 and $0<\delta<N$, then $\K^{N}_{M,L}$ maps boundedly $L^p_{\delta}$ to itself.
\end{lem}

By applying H\"older's inequality we deduce the following  pointwise estimate of the operators $K_{M,L}^N$, which will be often used in the forthcoming sections.

\begin{lem} \label{lem:HoldK} Let $N\ge 0$, $\tau>0$, $p\geq 1$ and $0<\e<N+\tau$. Then $\displaystyle{\left(\K^{N+\tau}_{N-n+1,n-1/2}(|\psi|)\right)^p\lesssim \K^{(N+\tau-\e)p}_{Np-n+1,n-1/2}(|\psi|^p).}$
\end{lem}

In the next lemma we state some differentiation formulas for both operators $\mathcal{P}^N$ and $\mathcal{K}^N$.

\begin{lem} \label{lem:derPN}
Let $N\ge 0$, $\alpha\in\N^n$ and $k=|\alpha|$.
\begin{enumerate}
\item \label{item:derPN1} If $\psi\in \mathcal{C}^{k}(\bar \B)$, then 
$\p_\alpha \mathcal{P}^{N}(\psi)=\mathcal{P}^{N+k}(\p_\alpha\psi)$.
\item \label{item:derPN2} If $\psi\in \mathcal{C}^{k+1}(\bar \B)$, then $\p_\alpha \mathcal{K}^{N}(\bar\p\psi)=\mathcal{K}^{N+k}(\bar\p\p_\alpha\psi)$.
\end{enumerate}
\end{lem}

\begin{proof} These results are well known (see, for instance, \cite[\S\,5]{Br-Or}). For the sake of completeness, we give a  brief sketch of the proof. For $N>0$, {\eqref{item:derPN1}}~follows
 from the equation $\frac{\p }{\p z_j}\mathcal{P}^N(w,z)=$ $\frac{\p }{\p w_j}\mathcal{P}^{N+1}(w,z)$ and integration by parts, while~{\eqref{item:derPN2}} is just a direct consequence of~{\eqref{eqn:rep}} and~{\eqref{item:derPN1}}.  

The case $N=0$ is deduced from the corresponding formulas for $N>0$ by taking $N\searrow 0$.
\end{proof}

\begin{rem} \label{rem:derPN} The above differentiation formulas will be applied to functions $\psi=\varphi f$ where $\varphi\in G^{\tau_0}_{\tau,k}$ and $f\in B^\infty_s$, $s>0$. The validity of the formulas in this more general setting can be shown by applying Lemma~{\ref{lem:derPN}} to $\psi_r(z)=\psi(rz)$ and the dominated convergence theorem.
\end{rem}

Now we state some regularity properties related to the integral operator $\mathcal{P}^N$. 
\begin{prop} \label{prop:contG} \blankline
\begin{enumerate}
	\item If $0<\delta<N$ and $1\le p<\infty$, then $\mathcal{P}^N$ maps continuously $L^p_\delta$  in  $B^p_{-\delta}$.
	\item If $N\ge 0$, then $\mathcal{P}^N$ maps continuously $G^{\tau_0}_{\tau,k}$ in  $B^\infty_\tau$.
	\item If $N=0$, then $\mathcal{P}$ maps continuously $\Lambda_{\tau}$ in  $B^\infty_\tau$.
\end{enumerate}
\end{prop}

\begin{proof} The proof of (i) can be found in~{\cite[Theorem 2.10]{Zhu}.}
 The proof of (ii) reduces to show that every $\psi\in G^{\tau_0}_{\tau,k}$ satisfies  
$$
|\p^k\mathcal{P}^N(\psi)(z)|=
|\mathcal{P}^{N+k}(\p^k\psi)(z)|\lesssim K^{N+\tau}_{n+N+k,0}(1)(z)\lesssim (1-|z|^2)^{\tau-k},
$$ 
 which follows from Lemmas~{\ref{lem:derPN}} and~{\ref{lem:estK}.} 
 Assertion (iii) can be found in \cite[\S\, 6.4]{Ru}.  
\end{proof}

\begin{prop} \label{prop:contTN}
If $\varphi\in G^{\tau_0}_{\tau,k}$, then $\T^N_\varphi$ maps boundedly $B^1_{-N}$ to itself, for $N>0$, and $H^1$ to 
 itself, for $N=0$.
 \end{prop}
 
 \begin{proof} The second assertion is a consequence of $G^{\tau_0}_{\tau,k}\subset \Lambda_{\tau_0}$ and {\cite[Theorem 6.5.4]{Ru},} so let us prove the first one.
Assume $N>0$. By Proposition~{\ref{prop:GLip0},} $\varphi\in\Lambda_{\tau_0}(\B)$, so 
$|\varphi(w)-\varphi(z)|\lesssim|w-z|^{\tau_0}\lesssim|1-\bar w z|^{\tau_0/2}$. Then, since 
$\T^N_\varphi(f)(z)=\T^N_{\varphi-\varphi(z)}(f)(z)+\varphi(z)f(z)$,
 we have that 
$$
|\T^N_\varphi(f)|\lesssim\K^N_{n+N-\tau_0/2,0}(|f|)+|f|.
$$
By Fubini's Theorem and Lemma~{\ref{lem:estK},}  $\|\K^N_{n+N-\tau_0/2,0}(|f|)\|_{L^1_N}\lesssim\|f\|_{L^1_N}$.
 Therefore $\|\T^N_\varphi(f)\|_{L^1_N}\lesssim\|f\|_{L^1_N}$, and the proof is complete.
 \end{proof}

\section{Toeplitz operators with symbols in $G^{\tau_0}_{\tau,k}$}

%\textbf{JO HO POSARIA AL FINAL DE TOT, NOMES ES PER RECORDAR QUE CAL AFEGIR-HO:  
%For instance assume that $\varphi(e_1)=0$ where $e_1=(1,0,...,0)$. Now consider the function $f(z)=(1-z_1)^{\tau-\e}$, $0<\e<\tau$. Then, for $\e$ small enough, the function $\varphi f\in\Lambda_\tau$, and therefore $\mathcal{P}(\varphi f)\in \Lambda_\tau$. But $f\notin B^\infty_\tau$.Of course the conclusion in  Theorem \ref{thm:IntrLip} $f\in\Lambda_\tau$ is sharp, because if $\varphi\equiv1$, then $\T_\varphi(f)=f$.}

%\textbf{MODIFICAR?????}
In this section we state a general theorem from which we will deduce the results stated in the introduction.
 The proof of this general theorem will be postponed to the next section.

Observe that if the functions $\varphi\in G^{\tau_0}_{\tau,k}$ and $f\in B^1_{-N}$, $N\ge0$, satisfy the equation $\T^N_{\varphi}(f)=h\in B^\infty_\tau$, then, taking into account Remark~{\ref{rem:thm:rep}}, 
formula~{\eqref{eqn:rep}} gives that 
 \begin{equation} \label{eqn:form210}
 \varphi f=\K^N(f\bar\p\varphi)+h.
 \end{equation}
Note that, by \eqref{eqn:estimate:kernel:K}, $$|\K^N(f\bar\p\varphi)|\lesssim \K^N_{N-n+1,n-1/2}(|f|\tilde \varphi) \lesssim \K^{N+\tau_0}_{N-n+1,n-1/2}(|f|)$$
and therefore by~{\eqref{estimate:kernel:K}}, $\K^N(f\bar\p\varphi)$ is pointwise defined even if $f\in B^1_{-N_0}$, for some $N<N_0<N+\tau_0$.  
This fact and the inclusion  $H^1\subset B^1_{-N_0}$ for any $N_0>0$, allow us to unify the proofs of
 Theorems~{\ref{thm:kerNInt}}  and~{\ref{thm:ker0Int},} using the following result:
 
\begin{thm}\label{thm:kerT1} Let $N\ge0$
and let $\varphi\in G^{\tau_0}_{\tau,k}$ be such that $0\notin\varphi(\bB)$.
 If  $0<N_0<N+\tau_0$,  $f\in B^1_{-N_0}$ and $h\in B_\tau^\infty$ satisfy \eqref{eqn:form210},
 then $f\in B^\infty_{\tau}$ and $\|f\|_{B_\tau^\infty}\lesssim\|f\|_{B^1_{-N_0}}+\|h\|_{B^\infty_\tau}$.
\end{thm} 

Now we easily deduce Theorems~{\ref{thm:kerNInt}}  and~{\ref{thm:ker0Int}} all at once:

%The following theorem is a  consequence of Theorem~{\ref{thm:kerT1}.} 

\begin{thm} \label{thm:kerT2}
 Let $\varphi\in G^{\tau_0}_{\tau,k}$ be such that $0\notin \varphi(\bB)$.
\begin{enumerate}
 \item If $f\in B^1_{-N}$ and $\T^N_{\varphi}(f)\in B^\infty_\tau$, for some $N>0$, then $f\in B^\infty_{\tau}$
 and $\|f\|_{B_\tau^\infty}\lesssim\|f\|_{B^1_{-N}}+\|h\|_{B^\infty_\tau}$.
 \item If $f\in H^1$ and $\T_{\varphi}(f)\in B^\infty_\tau$, then $f\in B^\infty_{\tau}$
  and $\|f\|_{B_\tau^\infty}\lesssim\|f\|_{H^1}+\|h\|_{B^\infty_\tau}$.
\end{enumerate}
\end{thm} 

\begin{proof} As we pointed out at the beginning of the section, if $\T^N_{\varphi}(f)=h\in B^\infty_\tau$, $N\ge0$,
 then $\varphi$ and $f$ satisfy \eqref{eqn:form210}.
 Therefore (i) directly follows from Theorem~{\ref{thm:kerT1}} (case $N>0$).
 By Proposition \ref{prop:embed}, $H^1\subset B^1_{-t}$, for every $t>0$, and, in particular, $H^1\subset B^1_{-N_0}$, for every $0<N_0<\tau_0$,
 so (ii) also follows from Theorem~{\ref{thm:kerT1}} (case $N=0$).   
\end{proof}

As an immediate consequence of Theorem~{\ref{thm:kerT2}} we obtain the following corollaries.

\begin{cor}\label{cor:kerT20}
Let $\tau>0$ and assume that $\varphi$ satisfy that $0\notin\varphi(\bB)$, and one of the following  conditions:
\begin{enumerate}
\item $n=1$ and $\varphi\in \Lambda_\tau$.
\item $n>1$, $\tau>\frac12$ and $\varphi\in \Lambda_\tau$.
\item $n>1$, $\tau\leq \frac12$ and $\varphi\in \Gamma_\tau$.
\end{enumerate}
If $f\in H^1$ and $\T_{\varphi}(f)\in B^\infty_\tau$, then $f\in B^\infty_{\tau}$.
\end{cor}
\begin{proof}
This is a consequence of Theorem \ref{thm:kerT2} and Corollaries \ref{cor:Glip1} and \ref{cor:Glip2}.
\end{proof}
\begin{cor}  Let $\varphi\in G^{\tau_0}_{\tau,k}$ be such that $0\notin \varphi(\bB)$.
\begin{enumerate}
 \item  If $N>0$ and $f\in B^1_{-N}$ satisfies that $\varphi f\in G_{\tau,k}^{\tau_0}+\ker \mathcal{P}^N$,
        then $f\in B^\infty_{\tau}$.
 \item If $f\in H^1$ satisfies that $\varphi f\in G_{\tau,k}^{\tau_0}+\ker \mathcal{P}$,
        then $f\in B^\infty_{\tau}$.
\end{enumerate}
\end{cor}

\begin{proof} 
 This is a consequence of Theorem~{\ref{thm:kerT2}} and Proposition~{\ref{prop:contG}(ii)}.
\end{proof}

\begin{cor} If $\varphi\in G^{\tau_0}_{\tau,k}$,
then $\ker(\T_{\varphi}^N-\lambda\I)\subset B^\infty_\tau$, for any $\lambda\in\C\setminus\varphi(\bB)$ and $N\ge0$.
In particular, $\ker \T^N_\varphi\subset B^\infty_\tau$, whenever $0\notin\varphi(\bB)$.
\end{cor}

\begin{proof} 
Since $\T_{\varphi}^N-\lambda\I=\T_{\varphi-\lambda}^N$, it directly follows from Theorem~{\ref{thm:kerT2}}.
\end{proof}

\begin{rem} If the condition $0\notin \varphi(\bB)$ is omitted, then $\ker\T^N_\varphi$ is not necessarily
 contained in $B^\infty_\tau$. For $n>1$, this result follows by taking $\varphi(z)=\bar z_1$, and observing that
 $\ker\T^N_\varphi$ contains any function in $B^1_{-N}$, if $N>0$, ($H^1$, if $N=0$), which does not depend on the
 first variable.  For $n=1$ we may consider the symbol $\varphi(z)=\bar z^{\,m+1}(1-z)^{m+\alpha}$ and the function 
$f(z)=(1-z)^{-\alpha}$, where $0<\alpha<1$  and $m\in\N$ such that $m+\alpha\ge\tau$, which satisfy
 $\varphi\in G^{\tau_0}_{\tau,k}$ and $f\in\ker\T^N_\varphi\setminus B^{\infty}_{\tau}$. 
\end{rem}

Now we extend Corollary \ref{cor:kerT20}.

\begin{thm} \label{thm:LipE}
 Let $\tau>0$ and let $\varphi\in \Lambda_\tau$ be a non-vanishing function on $\bB$.
 If $f\in H^1$ and $\T_\varphi(f)\in B^\infty_\tau$, then $f\in B^\infty_\tau$.
\end{thm}

\begin{proof}
If $\tau>1/2$, the result is just a consequence of Corollary~{\ref{cor:Glip1}} and part (ii) of Theorem~{\ref{thm:kerT2}.}

Now assume that $\tau\le 1/2$. Since $|w-z|^2\le 2|1-\bar w z|$, we have that $\Lambda_\tau\subset \Gamma_{\tau/2}$.
 And then Corollary~{\ref{cor:Glip2}} and part (ii) of Theorem~{\ref{thm:kerT2}} show that $f\in B^\infty_{\tau/2}$.
Thus $|\p f(z)|\lesssim (1-|z|^2)^{\tau/2-1}$, but we want to prove that $|\p f(z)|\lesssim (1-|z|^2)^{\tau-1}$,
 or equivalently $|\Phi(z)\p f(z)|\lesssim (1-|z|^2)^{\tau-1}$,
 $\Phi$ being the harmonic extension of $\varphi$ to $\B$. (Recall that, since $0\not\in\varphi(\bB)$, there is $0<r<1$ so
 that $|\Phi(z)|\simeq 1$ for $r\le|z|\le 1$.)

In order to show the estimate note that $|d\Phi(z)|\lesssim (1-|z|^2)^{\tau-1}$, which implies that
 $|\Phi(z)-\Phi(w)|\lesssim|z-w|^{\tau}\lesssim|1-\bar w z|^{\tau/2}$, for $z,w\in \B$.
 On the other hand, since $f\in B^\infty_{\tau/2}$, $\p_jf\in B^1_{-1}$ so $\p_jf=\mathcal{P}^1(\p_jf)$ and therefore 
$$
\Phi(z) \p_j f(z)=\mathcal{P}^1((\Phi(z)-\Phi)\p_jf)(z)+\mathcal{P}^1(\p_j(\Phi f))(z)-\mathcal{P}^1(f\p_j\Phi)(z).
$$
By Lemma~{\ref{lem:derPN},} $\mathcal{P}^1(\p_j(\Phi f))=\p_j\T_{\varphi}f$. Hence
$$
|\Phi(z)\p_j f(z)| \lesssim \K^{\tau/2}_{n+1-\tau/2,0}(1)(z) + (1-|z|^2)^{\tau-1} + \K^{\tau}_{n+1,0}(1)(z),
$$
and then Lemma~{\ref{lem:estK}} shows that $|\Phi(z)\p_j f(z)|\lesssim (1-|z|^2)^{\tau-1}$.
\end{proof}

Since $\mathcal{P}$ maps $\Lambda_\tau$ to $B^\infty_\tau$, we deduce 

\begin{cor}
 If $f\in H^1$ and $\varphi\in \Lambda_\tau$ satisfy  $0\notin\varphi(\bB)$
 and $\varphi f\in \Lambda_\tau+\ker \mathcal{P}$, then $f\in B^\infty_{\tau}$.
\end{cor}

Now we obtain Corollary~{\ref{cor:Intarg}}:

\begin{cor}\label{cor:Intarg:generalized}
 If $f,g\in H^1\setminus\{0\}$ satisfy $\varphi=\bar g/f\in \Lambda_{\tau}$ and $0\notin\varphi(\bB)$,
 then $f,g\in B^\infty_\tau$.
 In particular,  if $f\in H^1\setminus\{0\}$ and its argument function $\bar f/f$ is in $\Lambda_{\tau}$,
 then $f\in B^\infty_\tau$.
\end{cor}

\begin{proof}
 Since $\T_\varphi(f)=\mathcal{P}(\bar g)=\overline{g(0)}\in B^\infty_\tau$, Theorem~{\ref{thm:LipE}} 
shows that $f\in B^\infty_\tau$. Therefore $\bar g=f\varphi\in\Lambda_{\tau}$ and
 hence $g\in B^\infty_\tau$.
\end{proof}

\section{Proof of Theorem~{\ref{thm:kerT1}}}

%\textbf{MILLORAR AQUESTA INTRODUCCIO????}

This section is devoted to the proof of Theorem~{\ref{thm:kerT1}}.
 It is splitted into three steps composed of several lemmas that will give succesive improvements on the
 regularity of the solutions to the equation $\T^N_{\varphi}(f)=h$. 
First we will show that  any solution $f$ to  \eqref{eqn:form210} which is in $B^1_{-{N_0}}$ is in fact in any
 $B^1_{-t}$, $t>0$. Then we will obtain that the solution is in $B_{-t}^\infty$ for any $t>0$,
 and finally we will deduce that it is in $B_\tau^\infty$.
 
 Throughout this section we will assume that $\varphi$ and $h$ satisfy the hypotheses of
 Theorem~{\ref{thm:kerT1}.} 
Since  $|\varphi(\z)|\ge \rho>0$ on $\bB$, we can choose $r_0$ such that $|\varphi(z)|\ge \rho/2>0$ on the corona  $C=\{\,z\in \B\,:\,r_0\le|z|\le1\,\}$. 
 Let $\chi$ be a real $\mathcal{C}^\infty$-function on $\C^n$
 supported on the corona $C_0=\{\,z\in \B\,:\,r_0\le |z|\le 1+r_0\,\}$, such that $0\le\chi\le1$ and $\chi\equiv1$ on a neighborhood of $\bB$. Then~{\eqref{eqn:form210}} shows that
\begin{equation} \label{eqn:form21} 
f=\dfrac{\chi}{\varphi}\K^N(f\bar\p\varphi)+\dfrac{\chi}{\varphi}h+(1-\chi)f.
\end{equation}

The function  $(1-\chi)f$ is a $\mathcal{C}^\infty$ function with compact support on $\B$.
It is easy to prove that $\dfrac{\chi}{\varphi}\in G^{\tau_0}_{\tau,k}$,
 and so $\dfrac{\chi}{\varphi}h\in G^{\tau_0}_{\tau,k}$, by Proposition~{\ref{prop:prodGs}.} Therefore
 $\displaystyle{(1-\chi)f+\dfrac{\chi}{\varphi}h\in G^{\tau_0}_{\tau,k}}$ and 
\begin{equation}\label{eqn:regh}
 \|(1-\chi)f+\dfrac{\chi}{\varphi}h\|_{ G^{\tau_0}_{\tau,k}}\le C_\varphi\left(\|f\|_{B^1_{-N_0}}+\|h\|_{B^\infty_\tau}\right)
\end{equation}  
Hence, in order to prove that
$f\in B^\infty_\tau$, we have just to show that 
$\dfrac{\chi}{\varphi}\K^N(f\bar\p\varphi)\in G^{\tau_0}_{\tau,k}$.
\vspace*{12pt}

%And this will be a consequence of the fact that 
%$|\K^N(f\bar\p\varphi)|\lesssim \K^{N+\tau_0}_{N-n+1,n-1/2}(|f|) $ and that $\K^{N+\tau_0}_{N-n+1,n-1/2}$ is a kernel of type $\tau_0>0$. These facts will be used to succesively obtain improvements in the regularity of the solutions to the equation $\T^N_{\varphi}(f)=h$. Using these observations and \eqref{eqn:form21} we prove first that any solution $f$ to  \eqref{eqn:form210} which is in $B^1_{-N_0}$ and $-N_0<-N_0+k\tau_0<0$, is also in $B^1_{-N_0+\tau_0}$, and repeating the argument in $ B^1_{-N_0+2\tau_0}$. And consequently  $f\in B^1_{-t}$ for all $t<0$. The same type of arguments permit us to prove that there exists a sequence $1<p_1<p_2<...<p_k<\infty$ such that if $f\in B^1_{-t}$ and satisfies \eqref{eqn:form210}, then 
% $f\in B^{p_1}_{-t}$, and repeating the argument also $f\in B^{p_2}_{-t}$. Hence, $f\in B^{\infty}_{-t}$. Finally, we show that if $f\in B^{\infty}_{-t}$ satisfies \eqref{eqn:form210}, then $f\in B^{\infty}_{-t+\tau_0}$ and, as a consequence, $f\in B^{\infty}_{\tau}$.
% 

\noindent
{\bf\sc Step 1.} The first couple of lemmas will show that $f\in B^1_{-t}$, for any $t>0$.

\begin{lem}\label{lem1:thm:kerT1} 
Let $f\in B^1_{-s}$, for some $0<s<N+\tau_0$, and assume it satisfies~{\eqref{eqn:form210}.}
\begin{enumerate}
 \item  If $s\le\tau_0$ then $f\in B^1_{-t}$, for every $t>0$.
 \item  If $s>\tau_0$ then $f\in B^1_{-(s-\tau_0)}$.
\end{enumerate}
\end{lem}

\begin{proof}
First note that~{\eqref{eqn:estimates:P:K}} shows that 
$$
|\K^N(f\bar\p\varphi)|=|\K^N(\bar\p(f\varphi))|\lesssim \K^N_{N-n+1,n-\frac12}(|f|\pphi),
$$
and so 
\begin{equation}\label{eqn:KN4}
|\K^N(f\bar\p\varphi)|\lesssim\|\varphi\|_{G^{\tau_0}_{\tau,k}}\K^{N+\tau_0}_{N-n+1,n-1/2}(|f|).
\end{equation}
By integrating and using Fubini's Theorem, for any $t>0$ we have that
$$
\|\K^N(f\bar\p\varphi)\|_{L^1_t}\lesssim
\|\varphi\|_{G^{\tau_0}_{\tau,k}}\int_{\B}|f(w)|g_t(w)\,d\nu(w),
$$
where 
$$
g_t(w)=(1-|w|^2)^{N+\tau_0-1}\int_{\B}\K^t_{N-n+1,n-1/2}(z,w)\,d\nu(z).
$$
Now Lemmas~{\ref{lem:estK}} and~{\ref{lem:poteomegas}} show that 
\begin{equation}\label{eqn:gestimate}
g_t(w)\lesssim(1-|w|^2)^{N+\tau_0-1}\omega_{t-N}(w)\lesssim(1-|w|^2)^{s-1},
\end{equation}
provided that $s\le t+\tau_0$. (recall that $s<N+\tau_0$). 
Therefore, if $s\le t+\tau_0$,
\begin{equation}\label{eqn:lem41}
\|\K^N(f\bar\p\varphi)\|_{L^1_t}\lesssim\|\varphi\|_{G^{\tau_0}_{\tau,k}}\|f\|_{L^1_s},
\end{equation}
so, by~{\eqref{eqn:form21}} and \eqref{eqn:regh}, $f\in B^1_{-t}$. 
We conclude that:\vspace*{4pt}

\noindent
(i) If $s\le\tau_0$ then $s\le t+\tau_0$ and~\eqref{eqn:gestimate} holds for every $t>0$, and hence
 $f\in B^1_{-t}$, for every $t>0$.
\vspace*{6pt}

\noindent
(ii) If $s>\tau_0$ then~\eqref{eqn:gestimate} holds for $t=s-\tau_0$, and consequently
$f\in B^1_{-(s-\tau_0)}$.
\end{proof}

\begin{lem}\label{lem2:thm:kerT1} 
If $f\in B^1_{-N_0}$ satisfies~{\eqref{eqn:form210}} then $f\in B^1_{-t}$, for every $t>0$.
\end{lem}

\begin{proof}
For $N_0\le\tau_0$ the result  follows directly from Lemma~{\ref{lem1:thm:kerT1} \textbf{(i)}}. So assume that $N_0>\tau_0$.
 Let $k$ be the greatest positive integer such that $k\tau_0<N_0$.  Then $k\tau_0<N_0\le(k+1)\tau_0$. Now, since $f\in B^1_{-N_0}$, Lemma~{\ref{lem1:thm:kerT1}\textbf{(ii)}} implies that $f\in B^1_{-(N_0-\tau_0)}$, so $f\in B^1_{-(N_0-2\tau_0)}$, \dots, so $f\in B^1_{-(N_0-k\tau_0)}$. But $0<N_0-k\tau_0\le\tau_0$ and therefore Lemma~{\ref{lem1:thm:kerT1}\textbf{(i)}} shows that $f\in B^1_{-t}$, for every $t>0$.
\end{proof}

\begin{rem}\label{rem:constantstep1}
Observe that  the above arguments, \eqref{eqn:form21} and \eqref{eqn:lem41}  give in particular the estimate
$\|f\|_{B^1_{-t}}\lesssim \|f\|_{B^1_{-N_0}}+\|h\|_{B^\infty_\tau}$.
\end{rem}
\vspace*{12pt}

\noindent
{\bf\sc Step 2.} 
 The next couple of lemmas will show that the function $f$ is in $B^\infty_{-t}$, for any $t>0$. We follow the ideas in \cite{KD}.
 
\begin{lem}\label{lem3:thm:kerT1} 
Let $f\in B^p_{-s}$, for some $1\le p<\infty$ and for every $s>0$.
 If $f$ satisfies~{\eqref{eqn:form210}} then $f\in B^q_{-s}$, for every $s>0$ and for every $q$ 
such that $p<q<\infty$ and $\frac1p-\frac{\tau_0}n<\frac1q$.
\end{lem}

\begin{proof}
If $-t<-s<0$, then the space $B_{-s}^p\subset B_{-t}^p$. Consequently, we only have to prove the lemma, for $s$ sufficiently small. Let $p<q<\infty$ and $0<\e<N+\tau_0$. Assume $f$ satisfies~{\eqref{eqn:form210}}. Then, as we have shown in the proof of Lemma~{\ref{lem1:thm:kerT1}}, \eqref{eqn:KN4} holds, and so Lemma~{\ref{lem:HoldK}} gives
$$
|\K^N(f\bar\p\varphi)|^q\lesssim 
\|\varphi\|_{G^{\tau_0}_{\tau,k}}^q \K^{(N+\tau_0-\e)q}_{Nq-n+1,n-\frac12}(|f|^q).
$$
By Proposition~{\ref{prop:embed}(iii)}, $B^p_{-s}\subset B^\infty_{-s-n/p}$ and 
 $|f(w)|\lesssim \|f\|_{B^p_{-s}}(1-|w|^2)^{-s-\frac{n}p}$, which implies that
$$
|f(w)|^q=|f(w)|^{q-p}|f(w)|^p\lesssim
\|f\|_{B^p_{-s}}^{q-p}(1-|w|^2)^{(p-q)(s+\frac{n}p)}|f(w)|^p,
$$
and, by integrating, we get
$$
\K^{(N+\tau_0-\e)q}_{Nq-n+1,n-\frac12}(|f|^q)\lesssim
\|f\|_{B^p_{-s}}^{q-p}\K^{N(\e,s)}_{M,L}(|f|^p),
$$
where   
$N(\e,s)=(N+\tau_0-\e)q+(p-q)\left(s+\frac{n}p\right)
=sp+(N-s)q+nq\left(\frac{\tau_0-\varepsilon}{n}-\frac1{p}+\frac1{q}\right)$,
$M=Nq-n+1$ and $L=n-1/2$.

Therefore 
$$
\|\K^N(f\bar\p\varphi)\|_{L^q_s}\lesssim
                  \|\varphi\|_{G^{\tau_0}_{\tau,k}}\|f\|_{B^p_{-s}}^{1-\frac{p}q}I_{\e,s},
\,\,\mbox{ where }\,\,I_{\e,s}=\|\K^{N(\e,s)}_{M,L}(|f|^p)\|_{L^1_{sq}}.
$$ 
Thus we only have to prove that $I_{\e,s}^q\lesssim\|f\|_{B^p_{-s}}^p$, for $\e,s>0$ small enough and $\frac1p-\frac{\tau_0}n<\frac1q$, because then the previous estimate shows that
$\|\K^N(f\bar\p\varphi)\|_{L^q_s}\lesssim\|\varphi\|_{G^{\tau_0}_{\tau,k}}\|f\|_{B^p_{-s}}$ and 
hence, by~{\eqref{eqn:form21}} and \eqref{eqn:regh}, we conclude that $f\in B^q_{-s}$.
 In order to estimate $I_{\e,s}^q$, first apply Fubini's Theorem to get 
$$I_{\e,s}^q=\int_{\B}|f(w)|^p(1-|w|^2)^{N(\e,s)-1}
\left(\int_{\B}\K^{sq}_{M,L}(z,w)\,d\nu(z)\right)\,d\nu(w),$$
and since $n+sq-M-2L=(s-N)q$, then apply Lemma~{\ref{lem:estK}} to obtain 
$$
I_{\e,s}^q\lesssim \int_{\B}|f(w)|^p(1-|w|^2)^{N(\e,s)-1}\omega_{(s-N)q}(w)\,d\nu(w).
$$
Now we consider two cases:
\vspace*{6pt}

\noindent
{\sf Case $N=0$.} Then $(s-N)q=sq>0$ so 
$$
I_{\e,s}^q\lesssim\int_{\B}|f(w)|^p(1-|w|^2)^{N(\e,s)-1}\,d\nu(w)\le\|f\|^p_{B^p_{-s}},
$$ 
provided that $N(\e,s)>sp$, which holds for $\e,s>0$ small enough and $\frac1{p}-\frac1{q}<\frac{\tau_0}n$, since 
$$
\lim_{\e,s\searrow0}(N(\e,s)-sp)=nq\{\tfrac{\tau_0}n-(\tfrac1p-\tfrac1q)\}>0.
$$

\noindent
{\sf Case $N>0$.} Let $0<s<N$. Then $(s-N)q<0$ and so 
$$
I_{\e,s}^q\lesssim\int_{\B}|f(w)|^p(1-|w|^2)^{N(\e,s)+(s-N)q-1}\,d\nu(w)\le\|f\|^p_{B^p_{-s}},
$$ 
provided that $N(\e,s)+(s-N)q>sp$, which holds for $\e>0$ small enough and $\frac1{p}-\frac1{q}<\frac{\tau_0}n$, since 
$$
N(\e,s)+(s-N)q-sp=nq\{\tfrac{\tau_0-\e}n-(\tfrac1p-\tfrac1q)\}>0.
$$
And the proof is complete.
\end{proof}

\begin{lem}\label{lem4:thm:kerT1} 
Let $f\in B^1_{-s}$, for every $s>0$. If $f$ satisfies~{\eqref{eqn:form210}} then $f\in B^{\infty}_{-t}$, for every $t>0$.
\end{lem}

\begin{proof}
Let $k$ be the greatest positive integer such that $k\frac{\tau_0}{2n}<1$.  
Then $k\frac{\tau_0}{2n}<1\le(k+1)\frac{\tau_0}{2n}$. 
Let 
$$
p_j=\frac1{1-j\frac{\tau_0}{2n}}\qquad(j=0,\dots,k).
$$
Then $p_j\ge1$, for $j=0,\dots,k$, and 
$\frac1{p_j}=\frac1{p_{j-1}}-\frac{\tau_0}{2n}>\frac1{p_{j-1}}-\frac{\tau_0}{n}$, for $j=1,\dots,k$. 
Now, since $f\in B^{p_0}_{-s}$, for every $s>0$, and $f$ satisfies~{\eqref{eqn:form210}}, Lemma~{\ref{lem3:thm:kerT1}} shows that $f\in B^{p_1}_{-s}$ so $f\in B^{p_2}_{-s}$, \dots, so 
$f\in B^{p_k}_{-s}$, for every $s>0$. But $\frac1{p_k}-\frac{\tau_0}{2n}=1-(k+1)\frac{\tau_0}{2n}\le0$ and therefore Lemma~{\ref{lem3:thm:kerT1}} once again shows that $f\in B^q_{-s}$, for every $q>p_k$ and every $s>0$. Since $B^q_{-s}\subset B^{\infty}_{-s-\frac{n}q}$, by Proposition~{\ref{prop:embed}(\ref{item:embed3})}, we conclude that $f\in B^{\infty}_{-t}$, for every $t>0$.
\end{proof}

\begin{rem}\label{rem:constantstep2}
Observe that  the above arguments and   \eqref{eqn:form21} give the estimate
$\|f\|_{B^\infty_{-t}}\lesssim\|f\|_{B^1_{-t}}+\|h\|_{B^\infty_\tau}$.
\end{rem}
\vspace*{12pt}

\noindent
{\bf\sc Step 3.} 
In what follows we will finally deduce that $f\in B_\tau^\infty$.

\begin{lem}\label{lem5:thm:kerT1} 
 Let $f\in B^{\infty}_{-t}$, for every $t>0$.
 If $f$ satisfies~{\eqref{eqn:form210}} then $f\in H^{\infty}$.
\end{lem}

\begin{proof}
Since $f$ satisfies~{\eqref{eqn:form210}}, \eqref{eqn:KN4} holds, as we have shown in the proof of Lemma~{\ref{lem1:thm:kerT1}}. But
$$
\K^{N+\tau_0}_{N-n+1,n-\frac12}(|f|)\lesssim\|f\|_{B^{\infty}_{-t}}\K^{N+\tau_0-t}_{N-n+1,n-\frac12}(1)
$$
and, by Lemma~{\ref{lem:estK}}, $\K^{N+\tau_0-t}_{N-n+1,n-\frac12}(1)\lesssim\omega_{\tau_0-t}\lesssim 1$, for any $0<t<\tau_0$. Therefore 
$\|\K^N(f\bar\p\varphi)\|_{\infty}\lesssim\|\varphi\|_{G^{\tau_0}_{\tau,k}}\|f\|_{B^{\infty}_{-t}}$,
and, by~{\eqref{eqn:form210}} and  \eqref{eqn:regh} $f\in H^\infty$.
\end{proof}

In order to prove that $f\in B^\infty_\tau$, we will use the following formula.

\begin{lem} \label{lem:der35} 
Let $N\ge 0$, $\varphi\in G^{\tau_0}_{\tau,k}$ and $f\in H^\infty$. Then 
\begin{equation}\label{eqn:repder}
\varphi \p_\alpha f=\p_\alpha \mathcal{P}^{N}(\varphi f)-\sum_{\substack {\beta+\gamma=\alpha\\|\beta|<k}} c_{\alpha,\beta} \mathcal{P}^{N+k}(\p_\gamma\varphi \p_\beta f)+\K^{N+k}(\bar\p\varphi\p_\alpha f),
\end{equation}
for every $\alpha\in\N^n$, where $k=|\alpha|$ and $c_{\alpha,\beta}=\alpha!/(\beta!\gamma!)$.
\end{lem}

\begin{proof} 
First assume that  $\varphi\in \mathcal{C}^{k}(\bar \B)$ and $f\in H(\bar \B)$. By Theorem \ref{thm:rep} we have that
$\varphi \p_\alpha f=\mathcal{P}^{N+k}(\varphi \p_\alpha f)+\K^{N+k}(\bar\p\varphi\p_\alpha f)$.
Moreover,
$$
\varphi\p_\alpha f =
\p_\alpha(\varphi f)-\sum_{\substack{\beta+\gamma=\alpha\\|\beta|<k}}c_{\alpha,\beta}\p_\gamma\varphi \p_\beta f,
$$ 
and Lemma~{\ref{lem:derPN}} shows that $\mathcal{P}^{N+k}( \p_\alpha (\varphi f))=\p_\alpha \mathcal{P}^{N}(\varphi f)$.
 Hence we obtain~{\eqref{eqn:repder}.}

By a standard approximation argument (based on the dominated convergence theorem), we deduce the general case 
from the regular case just proved above.
\end{proof}

\begin{lem}\label{lem6:thm:kerT1} \blankline
\begin{enumerate}
\item  If $f\in H^{\infty}$ satisfies~{\eqref{eqn:form210}}, then for every $0<t\le\tau_0$, $f\in B^{\infty}_t$.
\item  Let $f\in B^{\infty}_s$, for some $0<s<\tau$. If $f$ satisfies~{\eqref{eqn:form210}}, then for every $0<t\le\min(\tau,s+\tau_0)$, $f\in B^{\infty}_t$.
\end{enumerate}
\end{lem}

\begin{proof}
Let $f$ and $t$ be as in either (i) or (ii), and assume $f$ satisfies~{\eqref{eqn:form210}}.
Let $k\in\Z$ and $\alpha\in\N^n$ such that $|\alpha|=k>\tau$.
We want to prove that $|\p_{\alpha}f(z)|\lesssim(1-|z|^2)^{t-k}$. Since
\begin{equation}\label{eqn:est:partial:alpha:f}
|\p_{\alpha}f|\le
\|\chi/\varphi\|_{\infty}\,|\varphi\p_{\alpha}f|+(1-\chi)|\p_{\alpha}f|,
\end{equation}
 we only need to estimate $|\varphi\p_{\alpha}f|$. 

Lemma~{\ref{lem:der35}}, {\eqref{eqn:form210}} and~{\eqref{eqn:estimates:P:K}} show that
\begin{equation}\label{eqn:est:varphi:partial:alpha:f}
|\varphi\p_{\alpha}f|\lesssim
      |\p_{\alpha}h| + \K^{N+k}_{n+N+k,0}(F_{\alpha}) + \K^{N+k}_{N+k-n+1,n-\frac12}(|\p_{\alpha}f|\pphi),
\end{equation}
where $\displaystyle{F_{\alpha}=\sum_{\substack{\beta+\gamma=\alpha\\|\beta|<k}} |\p_{\beta}f|\,|\p_{\gamma}\varphi|}$.
Note that we only need to prove that 
$$
|\p_{\alpha}f(w)|\pphi(w)\lesssim(1-|w|^2)^{t-k} 
\,\,\mbox{ and }\,\,\,\, 
 F_{\alpha}(w)\lesssim(1-|w|^2)^{t-k},
$$
because then~{\eqref{eqn:est:varphi:partial:alpha:f}} and Lemma~{\ref{lem:estK}} show that
\begin{eqnarray*}
|(\varphi\p_{\alpha}f)(z)|
&\lesssim& |\p_{\alpha}h(z)|
+\K^{N+t}_{N+k-n+1,n-\frac12}(1)(z)+\K^{N+t}_{n+N+k,0}(1)(z)\\
&\lesssim& \omega_{t-k}(z)+\omega_{\tau-k}(z)=(1-|z|^2)^{t-k},
\end{eqnarray*}
and therefore, by~{\eqref{eqn:est:partial:alpha:f}}, we conclude that  
$|\p_{\alpha}f(z)|\lesssim  (1-|z|^2)^{t-k}.$

\vspace*{6pt}
\noindent
(i) Let $f\in H^{\infty}$. Then $|\p_{\beta}f(w)|\lesssim\|f\|_{\infty}(1-|w|^2)^{-|\beta|}$, for every multiindex $\beta$. So 
$$
|\p_{\alpha}f(w)|\pphi(w)\lesssim \|\varphi\|_{G^{\tau_0}_{\tau,k}}\|f\|_{\infty}(1-|w|^2)^{\tau_0-k}
                         \lesssim \|\varphi\|_{G^{\tau_0}_{\tau,k}}\|f\|_{\infty}(1-|w|^2)^{t-k},
$$
for every $t\in\R$ such that $t\le\tau_0$. 
Next Lemma~{\ref{lem:poteomegas}} shows that
\begin{equation*}\begin{split}
F_{\alpha}(w)
&\lesssim\|\varphi\|_{G^{\tau_0}_{\tau,k}}\|f\|_{\infty}\sum_{i=0}^{k-1}(1-|w|^2)^{-i}\omega_{\tau-k+i}(z)
\\&\lesssim\|\varphi\|_{G^{\tau_0}_{\tau,k}}\|f\|_{\infty}(1-|w|^2)^{t-k},
\end{split}\end{equation*}
for every $t\in\R$ such that $t<1$ and $t\le\tau$.

\vspace*{6pt}
\noindent
(ii) Let $f\in B^{\infty}_s$, for some $0<s<\tau$. 
Then $|\p_{\beta}f(z)|\lesssim\|f\|_{B^{\infty}_s}\omega_{s-|\beta|}(z)$, for every multiindex $\beta$, so
Lemma \ref{lem:poteomegas} gives that 
\begin{equation*}\begin{split}
|\p_{\alpha}f(w)|\pphi(w)
&\lesssim\|f\|_{B^{\infty}_s}\|\varphi\|_{G^{\tau_0}_{\tau,k}}(1-|w|^2)^{s+\tau_0-k}\\ &\le\|f\|_{B^{\infty}_s}\|\varphi\|_{G^{\tau_0}_{\tau,k}}(1-|w|^2)^{t-k},
\end{split}\end{equation*}
for every $t\in\R$ such that $t\le s+\tau_0<s+1$, and Lemma~{\ref{lem:prodomegas} } shows that
\begin{eqnarray*}
F_{\alpha}(w)
&\lesssim& \|\varphi\|_{G^{\tau_0}_{\tau,k}}\|f\|_{B^{\infty}_s}S^{s,\tau}_{k-1,k}(w) \\
&\lesssim& \|\varphi\|_{G^{\tau_0}_{\tau,k}}\|f\|_{B^{\infty}_s}\{\omega_{s+1-k}(w)+(1-|w|^2)^{\tau-k}\} \\
&\lesssim& \|\varphi\|_{G^{\tau_0}_{\tau,k}}\|f\|_{B^{\infty}_s}(1-|w|^2)^{t-k},
\end{eqnarray*}
for every $t\in\R$ such that $t\le s+1$ and $t\le\tau$.
\end{proof}

\begin{lem}\label{lem7:thm:kerT1} 
If $f\in H^{\infty}$ satisfies~{\eqref{eqn:form210}} then $f\in B^{\infty}_\tau$.
\end{lem}

\begin{proof}
 Let $f\in H^{\infty}$. If $\tau=\tau_0$ there is nothing to prove by Lemma~{\ref{lem6:thm:kerT1} (i)}. 
So assume that $\tau_0<\tau$, and let $k\ge1$ be the greatest integer such that $k\tau_0<\tau$.
Since $f\in H^{\infty}$,  Lemma~{\ref{lem6:thm:kerT1} (i)} shows that $f\in B^{\infty}_{\tau_0}$, and then
Lemma~{\ref{lem6:thm:kerT1} (ii)} implies that $f\in B^{\infty}_{2\tau_0}$, so $f\in B^{\infty}_{3\tau_0}$,
... , so $f\in B^{\infty}_{k\tau_0}$, and hence $f\in B^{\infty}_{\tau}$.
\end{proof}

\begin{rem}\label{rem:constantstep3}
Observe that  the above arguments and   \eqref{eqn:form21} give the estimate
$\|f\|_{B^\infty_{\tau}}\lesssim \|f\|_{B^\infty_{-t}}+\|h\|_{B^\infty_\tau}$. 
\end{rem}

\begin{rem} The remarks \ref{rem:constantstep1},  \ref{rem:constantstep2} and \ref{rem:constantstep3}
 show that 
$$
\|f\|_{B^\infty_{\tau}}\lesssim \|f\|_{B^1_{-N_0}}+\|h\|_{B^\infty_\tau}.
$$
 
Note that the opposite estimate is always fulfilled.
 This follows from the continuous embedding $B^\infty_\tau\subset B^1_{-N_0}$, and 
the estimate $\|h\|_{B^\infty_{\tau}}\le \|\T^N_\varphi\|\|f\|_{B^\infty_{\tau}}$. 

Therefore
$$
\|f\|_{B^\infty_{\tau}}\approx\|f\|_{B^1_{-N}}+\|h\|_{B^\infty_\tau}.
$$
\end{rem}

%%%%%%%%%%%%%%%%%%%%%%%%%%%%%%%%%%%%%%%%%%%%%%%%%%%%%%

\end{document}